\documentclass[lettersize,journal]{IEEEtran}
\usepackage{cite}
\usepackage{amsmath,amssymb,amsfonts}
\usepackage{graphicx}
\usepackage{textcomp}
\usepackage{float}
\usepackage{multicol}
\usepackage{amsmath}
\usepackage{amsthm}
\usepackage{amssymb}
\usepackage{amsfonts}
\usepackage{graphicx}
\usepackage{subfigure}
\usepackage{url}
\usepackage{booktabs}
\usepackage{float}
\usepackage{color}
\usepackage{bm}
\usepackage{algorithm}
\usepackage{algpseudocode}

\usepackage{multicol}
\usepackage{multirow}

\newcommand{\col}{\hbox{col}}
\newtheorem{assmp}{\bf Assumption}

\newtheorem{rem}{\bf Remark}

\newtheorem{lem}{\bf Lemma}

\newtheorem{thm}{\bf Theorem}

\newtheorem{defin}{\bf Definition}

\newtheorem{problem}{\bf Problem}

\newcommand{\EQ}{\begin{eqnarray}}
	\newcommand{\EN}{\end{eqnarray}}
\newcommand{\EQQ}{\begin{eqnarray*}}
	\newcommand{\ENN}{\end{eqnarray*}}

\ifCLASSINFOpdf
\else
\fi

\begin{document}
	\title{Data-Driven Output Regulation via Internal Model Principle}

	\author{Liquan~Lin~and~Jie~Huang,~\IEEEmembership{Fellow,~IEEE}
		\thanks{This work was supported in part by the Research Grants Council of the Hong Kong Special Administration Region under grant No. 14201420.}
\thanks{The authors are  the Department of Mechanical and Automation Engineering, The Chinese University of Hong Kong, Hong Kong (e-mail: lqlin@mae.cuhk.edu.hk; jhuang@mae.cuhk.edu.hk. Corresponding author: Jie Huang.)}}
	

	\maketitle
	
	\begin{abstract}
	The data-driven techniques have been developed to deal with the output regulation problem of unknown linear systems by various approaches. In this paper,
we first extend an existing algorithm from single-input single-output linear systems to multi-input multi-output linear systems. Then, by separating the dynamics used in the learning phase and the control phase, we further  propose
an improved algorithm that significantly reduces the computational cost and  weaken the solvability conditions over the first algorithm.

	\end{abstract}
	

	\IEEEpeerreviewmaketitle

    \section{Introduction}

    Reinforcement learning (RL) or adaptive dynamic programming (ADP) has been effective in tackling diverse optimal control problems. Combining RL techniques with parameter identification techniques further leads to the data-driven method to handle LQR problem for uncertain linear systems.
    For example, \cite{vrabie2009adaptive} studied the LQR problem for partially unknown linear systems by developing a policy-iteration (PI) method to iteratively solve an algebraic Riccati equation online. \cite{jiang2012computational} further extended the result of \cite{vrabie2009adaptive} to completely unknown linear systems with PI-based method. Since the PI-based method  requires an initially stabilizing feedback gain to start the iteration process,  \cite{bian2016} proposed a value-iteration (VI) method to solve the LQR problem of unknown linear systems without requiring an initially stabilizing feedback gain.

    Recently, the data-driven approach has been further extended to deal with the output regulation problem (ORP) for uncertain linear systems.
    There are two approaches, namely, the feedforward design and internal model design, to handling the OPR. Both of these two approaches are model-based.
Based on the PI method proposed by \cite{jiang2012computational}, \cite{gao2016adaptive} initialized the study of the optimal ORP for linear systems with unknown state equation using the feedforward design approach. Subsequently,  \cite{liu2018adaptive} further considered  the PI-based method to address the optimal ORP for unknown SISO linear systems based on  the internal model principle.   \cite{jiang2022value} leveraged the VI method of \cite{bian2016} to solve the optimal ORP with the feedforward design method. More recently, \cite{gao2021reinforcement}  proposed a VI-based algorithm  to solve the cooperative optimal ORP for unknown linear multi-agent systems via the distributed internal model approach. However, the followers in \cite{gao2021reinforcement} are assumed to be SISO linear systems and the input-output transmission matrix of each follower is assumed to be zero.
 It is noted that the algorithm of \cite{gao2016adaptive} was refined in \cite{lin2023} that reduced the computational cost  and weakened the solvability conditions of the algorithm of \cite{gao2016adaptive}. The result of  \cite{lin2023} was further extended to the  adaptive optimal cooperative output regulation problem of multi-agent systems in \cite{lin2024}.

    %
%

    In this paper, we first extend the VI-based data-driven algorithm of \cite{gao2021reinforcement} from
     single-input single-output linear systems to multi-input multi-output linear systems. Besides, we allow the input-output transmission matrix to be non-zero.
 Like in \cite{gao2021reinforcement}, the resultant algorithm  comes down to iteratively solving a sequence of linear equations.
     For a linear system with $m$ inputs, $p$ outputs,  $n$-dimensional state vector, and $q$-dimensional exosystem, each equation in the sequence typically contains $\frac{(n+pq)(n+pq +1)}{2} + (m+q)(n+pq)$ unknown variables. Thus, the computational cost can be formable for a high-dimensional system.
   For this reason, by separating the dynamics used in the learning phase and the control phase, we further propose an improved algorithm that also leads to
   a sequence of linear equations but which typically contains $\frac{(n+pq)(n+pq +1)}{2} + (n+pq)m + nq$ unknown variables. If the input-output transmission matrix is equal to zero like in \cite{gao2016adaptive} and \cite{gao2021reinforcement}, the number of the unknown variables is further reduced to $\frac{(n+pq)(n+pq +1)}{2} + n(m + q)$.
   Moreover,    once  the first step of the iteration is completed, we can further reduce
   the number of the unknown variables to $\frac{(n+pq)(n+pq +1)}{2} $. Thus,
this improved algorithm  significantly reduces the computational cost and  also weakens the solvability conditions of the first algorithm.

    \indent \textbf{Notation} Throughout this paper, $\mathbb{R}, \mathbb{N} , \mathbb{N}_+$ and $ \mathbb{C}_-$ represent the sets of real numbers, nonnegative integers,  positive integers and the open left-half complex plane, respectively. $\mathcal{P}^n$ is the set of all $n\times n$ real, symmetric and positive semidefinite matrices. $||\cdot||$ represents the Euclidean norm for vectors and the induced norm for matrices. $\otimes$ denotes the Kronecker product.  For $b=[b_1, b_2, \cdots, b_n]^T\in \mathbb{R}^n$, $\text{vecv}(b)=[b_1^2,b_1b_2,\cdots,b_1b_n,b_2^2,b_2b_3,\cdots, b_{n-1}b_n,b_n^2]^T \in \mathbb{R}^{\frac{n(n+1)}{2}}$. For a symmetric matrix $P=[p_{ij}]_{n\times n}\in \mathbb{R}^{n\times n}$, $\text{vecs}(P)=[p_{11},2p_{12},\cdots,2p_{1n},p_{22},2p_{23},\cdots, 2p_{n-1,n}, p_{nn}]^T\in \mathbb{R}^{\frac{n(n+1)}{2}}$. For  $v\in \mathbb{R}^n$, $|v|_P=v^TPv$. For column vectors $a_i, i=1,\cdots,s$,  $\mbox{col} (a_1,\cdots,a_s )= [a_1^T,\cdots,a_s^T  ]^T,$ and, if  $A = (a_1,\cdots,a_s )$, then  vec$(A)=\mbox{col} (a_1,\cdots,a_s )$.
    For $A\in \mathbb{R}^{n\times n}$, $\sigma(A) $ denotes the set composed of all the eigenvalues of $A$.
    For matrices $A_i, i=1,\cdots,n$,
$\mbox{block~diag}(A_1,...,A_n)$ is the block diagonal matrix $
	\left[\begin{array}{ccc}
A_1&&\\
&\ddots&\\
&&A_n
\end{array}\right]$.
$I_n $ denotes the identity matrix of dimension $n$.

    \section{Preliminary} \label{sec2}

    In this section, we summarize both the PI-based approach proposed in \cite{jiang2012computational} and the VI-based approach proposed in \cite{bian2016} for solving the model-free LQR problem. {\color{black} The internal model approach for solving the output regulation problem is also presented based on \cite{huang2004}.}

  \subsection{Iterative Approach to LQR Problem} \label{sec2-1}

    Consider the following linear system:
    \begin{align} \label{lisys}
    	\dot{x}=Ax+Bu
    \end{align}
    where $x\in \mathbb{R}^n$ is the system state, $u\in\mathbb{R}^m$ is the input.   We  make the following assumption:

    \begin{assmp}\label{ass1}
    	The pair $(A,B)$ is stabilizable.
    \end{assmp}

    The  LQR problem for \eqref{lisys} is to find a control law $u=\bar{K}x$ such that the  cost $\int_{0}^{\infty}(x^T\bar{Q}x+u^T\bar{R}u)d\tau$ is minimized,
    where $\bar{Q}=\bar{Q}^T\geq0, \bar{R}=\bar{R}^T>0$, with $(A,\sqrt{\bar{Q}})$ observable.

   By \cite{kucera1972},  under Assumption \ref{ass1},   the following algebraic Riccati equation
    \begin{align}\label{areli}
    	A^T\bar{P}^*+\bar{P}^*A+\bar{Q}-\bar{P}^*B\bar{R}^{-1}B^T\bar{P}^*=0
    \end{align}
admits a unique positive definite solution $\bar{P}^*$, and the solution of
 the LQR problem of \eqref{lisys} is given by $\bar{K}^*=-\bar{R}^{-1}B^T\bar{P}^*$.

 As  \eqref{areli} is nonlinear,   an iterative approach for obtaining $\bar{P}^* $ is given by solving the following linear Lyapunov equations \cite{Kleinman}:
 \begin{subequations} \label{aleq1sin}
	\begin{align}
&0=A_k^T\bar{P}^P_{k} +\bar{P}^P_{k}A_k+\bar{Q}+(\bar{K}^P_{k})^T\bar{R}\bar{K}^P_{k} \\
     	&\bar{K}^P_{k+1}=-\bar{R}^{-1} B^T\bar{P}^P_{k}
   \end{align}
  \end{subequations}
  where $A_k=A+B\bar{K}^P_k$, $k=0,1,\cdots$, and $\bar{K}^P_{0}$ is such that $A_0 $ is a Hurwitz matrix. The algorithm \eqref{aleq1sin}  guarantees the following properties for $k\in \mathbb{N}$:
   \begin{enumerate}
   	\item $\sigma(A_{k}) \subset \mathbb{C}_-$;
   	\item $\bar{P}^*\leq \bar{P}^P_{k+1}\leq \bar{P}^P_k$;
   	\item $\lim\limits_{k\to \infty}\bar{K}^P_k=\bar{K}^*,\lim\limits_{k\to \infty}\bar{P}^P_k=\bar{P}^*$.
   \end{enumerate}

 \subsection{PI Method for Solving LQR Problem without Knowing $A$ and $B$} \label{sec2-2}

   Based on Kleinman's algorithm \eqref{aleq1sin}, \cite{jiang2012computational} proposed a model-free algorithm to obtain the approximate solution to \eqref{areli}. It is noted from \eqref{lisys} that
   \begin{align} \label{pidy}
   	\dot{x}=A_kx-B(\bar{K}^P_kx-u)
   \end{align}
   Thus,
   \begin{align}\label{piint}
   	|x(t+\delta t)|_{\bar{P}^P_k}-&|x(t)|_{\bar{P}^P_k}=\int_{t}^{t+\delta t}-|x|_{\bar{Q}+(\bar{K}^P_{k})^T\bar{R}\bar{K}^P_{k}}\notag \\
   	&+2x^T(\bar{K}^P_k)^T\bar{R}\bar{K}^P_{k+1}x-2u^T\bar{R}\bar{K}^P_{k+1}x d\tau
   \end{align}

   For convenience,  for any vectors $a\in \mathbb{R}^n,b \in \mathbb{R}^m$ and any integer $s\in \mathbb{N}_+$, define

   \begin{align}
   	\begin{split}\label{newdefi}
   		\delta _a=&[\text{vecv}(a(t_1))-\text{vecv}(a(t_0)), \cdots ,\\ &\text{vecv}(a(t_s))-\text{vecv}(a(t_{s-1}))]^T\\
   		\Gamma_{ab}=&[\int_{t_0}^{t_1}a\otimes b d\tau , \int_{t_1}^{t_2}a\otimes b d\tau, \cdots , \int_{t_{s-1}}^{t_s}a\otimes b d\tau]^T
   	\end{split}
   \end{align}

   Then, \eqref{piint} and \eqref{newdefi} imply
   \begin{align}\label{pilinear}
   	\Psi_k\begin{bmatrix}
   		\text{vecs}(\bar{P}^P_k)\\
   		\text{vec}(\bar{K}^P_{k+1})
   	\end{bmatrix}=\Phi_k
   \end{align}
   where $\Psi_k=[\delta_x, -2\Gamma_{xx}(I_n\otimes ((\bar{K}^P_k)^T\bar{R}))+2\Gamma_{xu}(I_n\otimes \bar{R})]$ and $\Phi_k=-\Gamma_{xx}\text{vec}(\bar{Q}+(\bar{K}^P_{k})^T\bar{R}\bar{K}^P_{k})$.
   By iteratively solving \eqref{pilinear}, one can finally obtain the approximate solution to \eqref{areli} without using $A$ and $B$.

   The solvability of \eqref{pilinear} is guaranteed by the following lemma \cite{jiang2012computational}.
   \begin{lem}
   For $k=0,1,\cdots$, 	the matrix $\Psi_k$ has full column rank if
   	\begin{align}\label{rankconpi}
   		\textup{rank}([\Gamma_{xx}, \Gamma_{xu}])=& \frac{n(n+1)}{2}+mn
   	\end{align}
   \end{lem}

   The iteration approach using \eqref{pilinear} is called PI method. A drawback of PI method is that it needs an initial stabilizing control gain $\bar{K}^P_{0}$ to start the iteration. The VI approach of \cite{bian2016} to be introduced in the next subsection does not need an  initial stabilizing control gain.

 \subsection{VI Method for Solving LQR Problem without Knowing $A$ and $B$} \label{sec2-3}

   To introduce the VI method proposed in \cite{bian2016}, let $\epsilon_k$ be a real sequence satisfying
   \begin{align}
    	\epsilon_k>0, \; \sum_{k=0}^{\infty}\epsilon_k=\infty, \; \sum_{k=0}^{\infty}\epsilon_k^2<\infty,
    \end{align}
let    $\{B_q\}_{q=0}^{\infty}$ be a collection of bounded subsets in $\mathcal{P}^{n}$ satisfying
    \begin{align}
    	B_q\subset B_{q+1}, \; q\in \mathbb{N}, \; \lim\limits_{q\to \infty}B_q=\mathcal{P}^{n}
    \end{align}
    and let $\varepsilon>0$  be some small real number for determining the convergence criterion.
   Then the solution to \eqref{areli} can be approximated by the following Algorithm \ref{vialg1}.
 \begin{algorithm}
    	\caption{{\color{black}Model-based} VI  Algorithm for solving \eqref{areli} \cite{bian2016}}\label{vialg1}
    	\begin{algorithmic}[1]
    		\State Choose $\bar{P}^V_0=(\bar{P}^V_0)^T>0$. $k,q\leftarrow 0$.
    		\Loop
    		\State $\tilde{P}^V_{k+1}\leftarrow \bar{P}^V_k+\epsilon_k (A^T\bar{P}^V_k+\bar{P}^V_kA-\bar{P}^V_kB\bar{R}^{-1}B^T\bar{P}^V_k+\bar{Q})$
    		\If{$\tilde{P}^V_{k+1}\notin B_q$}
    		\State $\bar{P}^V_{k+1}\leftarrow \bar{P}^V_0$. $q\leftarrow q+1$.
    		\ElsIf{$|\tilde{P}^V_{k+1}-\bar{P}^V_k|/\epsilon_k<\varepsilon$}
    		\State \Return $\bar{P}^V_k$ as an approximation to $\bar{P}^*$
    		\Else
    		\State $\bar{P}^V_{k+1}\leftarrow \tilde{P}^V_{k+1}$
    		\EndIf
    		\State $k \leftarrow k+1$
    		\EndLoop
    	\end{algorithmic}
    \end{algorithm}

    The following result is from \cite{bian2016}.
    \begin{thm}
    	Under Assumption \ref{ass1},  Algorithm \ref{vialg1} is such that $\lim\limits_{k\to \infty}\bar{P}^V_k=\bar{P}^*$.
    \end{thm}

    \begin{rem}
    	 Algorithm \ref{vialg1} is called VI method, which  starts the iteration from any positive definite matrix $\bar{P}^V_0$. Thus, it does not need to know an  initial stabilizing control gain.
   \end{rem}

   When $A$ and $B$ are unknown, Algorithm \ref{vialg1} cannot  be directly used. To circumvent this difficulty,
   let $\bar{H}^V_k=A^T\bar{P}^V_k+\bar{P}^V_kA, \bar{K}^V_k=-\bar{R}^{-1}B^T\bar{P}^V_k$ and let
    \begin{align}
   	\begin{split}\label{defi}
   		I_{xx}=&[\int_{t_0}^{t_1}\textup{vecv}(x)d\tau , \int_{t_1}^{t_2}\textup{vecv}(x) d\tau, \cdots , \int_{t_{s-1}}^{t_s}\textup{vecv}(x) d\tau]^T\\
   		I_{xu}=&[\int_{t_0}^{t_1}x\otimes \bar{R}u d\tau , \int_{t_1}^{t_2}x\otimes \bar{R}u d\tau, \cdots , \int_{t_{s-1}}^{t_s}x\otimes \bar{R}u d\tau]^T\\
   	\end{split}
   \end{align}
   Then,
  it is obtained from \eqref{lisys} that
   \begin{align}\label{intvipre}
   	|x(t+\delta t)|_{\bar{P}^V_k}-|x(t)|_{\bar{P}^V_k} =\int_{t}^{t+\delta t}\lbrack |x|_{\bar{H}^V_k}-2{u}^T\bar{R}\bar{K}^V_kx \rbrack d\tau
   \end{align}
\eqref{defi} and  \eqref{intvipre}  imply
   \begin{align}\label{vilinearpre}
   	\begin{bmatrix}
   		I_{xx}&-2I_{xu}
   	\end{bmatrix}\begin{bmatrix}
   		\text{vecs}(\bar{H}^V_k)\\
   		\text{vec}(\bar{K}^V_k)
   	\end{bmatrix}=\delta_{x}\text{vecs}(\bar{P}^V_k)
   \end{align}



The solvability of \eqref{vilinearpre} is guaranteed by the following lemma \cite{bian2016}.
\begin{lem}
	The matrix $[I_{xx},-2I_{xu}]$ has full column rank  if
	\begin{align}\label{rankconpi}
		\textup{rank}([I_{xx}, I_{xu}])=& \frac{n(n+1)}{2}+mn
	\end{align}
\end{lem}

  {\color{black} \begin{rem}
   	In Step 3 of Algorithm \ref{vialg1}, $\tilde{P}^V_{k+1}$ can be represented by $\tilde{P}^V_{k+1}:= \bar{P}^V_k+\epsilon_k (\bar{H}^V_k-(\bar{K}^V_k)^T\bar{R}\bar{K}^V_k+\bar{Q})$. Thus, once $\bar{H}^V_k$ and $\bar{K}^V_k$ are solved from \eqref{vilinearpre}, one can update $\bar{P}^V_{k+1}$ following Steps 2-12 of Algorithm \ref{vialg1}. Thus, by iteratively applying the solution of \eqref{vilinearpre} to Algorithm \ref{vialg1}, we can finally obtain the approximate solution to \eqref{areli} without knowing $A$ and $B$.
   \end{rem}}

    \subsection{Output Regulation Problem and Internal Model Approach}\label{sec2-4}
    In this subsection, we review some results on the output regulation theory by internal model approach based on \cite{huang2004}.

    Consider a class of continuous-time linear systems in the following form:
    \begin{align}
    \begin{split}\label{follower}
    	\dot{x}&=Ax+Bu+Ev\\
    	y&=Cx+Du\\
    	e&=Cx+Du+Fv
    \end{split}
    \end{align}
    where $x\in \mathbb{R}^n$, $u\in \mathbb{R}^m$ and $y\in \mathbb{R}^p$ are the state, control input and output of the system, respectively, and  $v\in \mathbb{R}^q$ is the exogenous signal  generated by the following  exosystem:
    \begin{align}\label{leader}
    	\dot{v}=Sv
    \end{align}
	Let $-Fv$ be the reference output generated by the exosystem. Then, $e\in \mathbb{R}^p$ represents the tracking error of the system.

	The dynamic state feedback output regulation problem is defined as follows.
	\begin{problem}\label{p1}
		Given  plant \eqref{follower} and exosystem \eqref{leader},  design a dynamic state feedback controller $u$ of the following form
		\begin{align}
			\begin{split}\label{dycontroller}
				u=&K_xx+K_zz\\
				\dot{z}=&{G}_1z+{G}_2e
			\end{split}
		\end{align}
		with $z\in \mathbb{R}^{n_z}$
		such that the closed-loop system is exponentially stable with $v$ set to zero and $\lim\limits_{t\to \infty}e(t)=0$.
	\end{problem}
	
	In addition to Assumption \ref{ass1}, the solvability of Problem \ref{p1} needs the following assumptions.
	
	\begin{assmp}\label{ass2}
		$S$ has no eigenvalues with negative real parts.
	\end{assmp}
	
	\begin{assmp}\label{ass3}
		For all $\lambda \in \sigma(S)$, $\textup{rank}(\begin{bmatrix}
			A-\lambda I & B\\C & D
		\end{bmatrix})=n+p$.
	\end{assmp}

	The concept of internal model is given  as follows:
	\begin{defin} \color{black}
		A pair of matrices $(G_1,G_2)$ is said to be a minimum $p-\textup{copy}$ internal model of the matrix $S$ if
		\begin{align*}
			G_1=\textup{blockdiag}\underbrace{[\beta,\cdots,\beta]}_{p-tuple}, \; G_2=\textup{blockdiag}\underbrace{[\sigma,\cdots,\sigma]}_{p-tuple}
		\end{align*}
		where $\beta$ is a constant square matrix whose characteristic polynomial equals the minimal polynomial of $S$, and $\sigma$ is a constant column vector such that $(\beta, \sigma)$ is controllable.
	\end{defin}

	Let  $({G}_1,{G}_2)$ be a minimum $p-copy$ internal model of $S$. We call the following  system the augmented system.
	\begin{align}
		\begin{split}\label{augsys}
			\dot{x}&=Ax+Bu+Ev\\
			\dot{z}&={G}_1z+{G}_2e\\
			e&=Cx+Du+Fv
		\end{split}
	\end{align}

	Let $\xi=\col (x,z)$, $Y=\begin{bmatrix}
    	A&0\\ {G}_2C&{G}_1
    \end{bmatrix}$ and $J=\begin{bmatrix}
    B\\ {G}_2D
    \end{bmatrix}$. Then the state equation of the augmented system \eqref{augsys} can be put into the following compact form:
     \begin{align}\label{augcom}
		\dot{\xi}=&Y\xi+Ju+\begin{bmatrix}
					E\\G_2F
				\end{bmatrix}v
				\end{align}

 Let  $K=[K_x,K_z]$. Then, applying  the static state feedback control law $u = K \xi$ to the augmented system \eqref{augcom} gives
	\begin{align}\label{dyxi}
		\begin{split}
				\dot{\xi}=&Y\xi+Ju+\begin{bmatrix}
					E\\G_2F
				\end{bmatrix}v\\
				=&(Y+JK)\xi+\begin{bmatrix}
				E\\G_2F
			\end{bmatrix}v.
		\end{split}
	\end{align}

 \begin{rem}
 Generically, the minimal polynomial of $S$ is equal to the characteristic polynomial of $S$. Thus, typically, $n_z = pq$ and hence the dimension of $\xi$ is typically equal to $(n+pq)$.
\end{rem}


    By Lemma 1.26 and Remark 1.28 of \cite{huang2004}, the solvability of Problem \ref{p1} is summarized by the following theorem.

    \begin{thm} \label{thm1}
    	Under Assumptions \ref{ass1}-\ref{ass3}, let $({G}_1,{G}_2)$ be a minimum $p-copy$ internal model of $S$. Then, the pair $(Y,J)$ is stabilizable.
   Let  $K$ be such that $Y+JK$ is Hurwitz. Then Problem \ref{p1} is solved by the dynamic state feedback controller \eqref{dycontroller}.
    \end{thm}

	Let  $K$ be such that $Y+JK$ is Hurwitz.  Then, by Lemma 1.27 of \cite{huang2004}, the following matrix equations admit a unique solution $(X,Z)$.
	 \begin{align}\label{regueq}
	 	\begin{split}
	 		XS=&AX+B(K_xX+K_zZ)+E\\
	 		ZS=&G_1Z+G_2[CX+D(K_xX+K_zZ)+F]\\
	 		0=&CX+D(K_xX+K_zZ)+F
	 	\end{split}
	 \end{align}
	
	 To see the role of the internal model, let $\tilde{x}=x-Xv, \tilde{z}=z-Zv$ and $\tilde{u}=u-Uv$ with $U=K_xX+K_zZ$. Then,  we have
	 \begin{align}\label{regueqx}
	 	\begin{split}
	 		\dot{\tilde{x}}=&A\tilde{x}+B\tilde{u}\\
	 		\dot{\tilde{z}}=&G_2C\tilde{x}+G_1\tilde{z}+G_2D\tilde{u}\\
	 		e=&C\tilde{x}+D\tilde{u}
	 	\end{split}
	 \end{align}
	 Let $\tilde{\xi}= \col (\tilde{x},\tilde{z})$. Then, \eqref{regueqx} can be put into the following form:
	 \begin{align}\label{tildexi}
	 	\dot{\tilde{\xi}}=Y\tilde{\xi}+J\tilde{u}
	 \end{align}
	 Thus, if $Y+JK$ is Hurwitz, then the control $\tilde{u}=K\tilde{\xi}$ is such that  $\lim\limits_{t\to \infty}e(t)=0$. That is,  $u=K\tilde{\xi}+Uv=K_xx+K_zz=K\xi$ solves Problem \ref{p1}.
	A stabilizing feedback gain $K$ can be obtained by solving the following LQR problem:
	 \begin{problem}\label{p2}
	 		\begin{align}\label{lqr}
	 			\begin{split}
	 			&\min_{\tilde{u}}\int_{0}^{\infty} (|\tilde{\xi}|_Q+|\tilde{u}|_R)dt\\
	 			&\textup{subject to \eqref{tildexi}}
	 			\end{split}
	 		\end{align}
	 		where $Q=Q^T\geq 0, R=R^T>0$ with $(Y,\sqrt{Q})$ observable.
	 \end{problem}
	
	 By Theorem \ref{thm1}, $(Y,J)$ is stabilizable. Thus, the following algebraic Riccati equation
	 \begin{align}\label{riccati}
	 	Y^TP^*+P^*Y+Q-P^*JR^{-1}J^TP^*=0
	 \end{align}
	 has a unique positive definite solution $P^*$. Therefore, the optimal solution to Problem \ref{p2} is given by $\tilde{u}^*=K^*\tilde{\xi}$ with a stabilizing feedback gain $K^*=-R^{-1}J^TP^*$.

	    \section{Model-free Output Regulation via Internal Model Principle and value iteration } \label{sec3}

When the matrices $A,B,C,D,E,F$ are unknown,
reference  \cite{liu2018adaptive} presented a PI-based method for solving the output regulation problem of linear SISO systems with $D=0$.
 Reference \cite{gao2021reinforcement} further presented the VI-based method for the cooperative output regulation problem for $N$ single-input and single-output linear systems, which contains the output regulation problem of the system \eqref{follower} as a special case with $N=1$.
In this section, we will first extend the VI method of \cite{gao2021reinforcement} to  the output regulation problem for multi-input, multi-output systems of the form \eqref{follower} with $D \neq 0$. Then, we further propose a modified scheme that will significantly reduce the computational cost of the first scheme and weaken the solvability condition. We assume that $\col (\xi,u)$ is  available.

    \subsection{VI-based Data-driven Algorithm} \label{viordi}

   Let
   \begin{align}\label{hkkk}
   	\begin{split}
   		H_k&=Y^TP_k+P_kY\\K_k&=-R^{-1}J^TP_k
   	\end{split}
   \end{align}
    From \eqref{dyxi} and \eqref{hkkk}, we have
    \begin{align}\label{intvi}
    		|\xi&(t+\delta t)|_{P_k}-|\xi(t)|_{P_k}\notag \\=&\int_{t}^{t+\delta t}\lbrack |\xi|_{H_k}-2u^TRK_k\xi+ 2v^T\begin{bmatrix}
    			E\\G_2F
    		\end{bmatrix}^TP_k\xi \rbrack d\tau
    \end{align}

Note that, for any $\xi\in\mathbb{R}^{n+n_z}$ and $H_k=H_k^T\in\mathbb{R}^{(n+n_z)\times(n+n_z)}$, we have
	\begin{align}\label{xihk}
	|\xi|_{H_k}&=(\xi\otimes\xi)^T\textup{vec}(H_k)= \textup{vecv}(\xi)^T\textup{vecs}(H_k)
	\end{align}
%
%
%

Let $$I_{\xi\xi}=[\int_{t_0}^{t_1}\textup{vecv}(\xi)d\tau , \int_{t_1}^{t_2}\textup{vecv}(\xi) d\tau, \cdots , \int_{t_{s-1}}^{t_s}\textup{vecv}(\xi)d\tau]^T.$$ Then
\eqref{xihk}	 together with \eqref{newdefi} and  \eqref{intvi} implies
    \begin{align}\label{vilinear}
    	\Psi_{VI}\begin{bmatrix}
    		\text{vecs}(H_k)\\
    		\text{vec}(K_k) \\
    		\text{vec}(\begin{bmatrix}
    			E\\G_2F
    		\end{bmatrix}^TP_k)
    	\end{bmatrix}=\Phi_{VI, k}
    \end{align}
    where
    \begin{align*}
    	\Psi_{VI}=&[\Psi_{VI}^1, \Psi_{VI}^2, \Psi_{VI}^3]\\
    	\Psi_{VI}^1=&I_{\xi\xi}\\
    	\Psi_{VI}^2=&-2\Gamma_{\xi u}(I_{n+n_z}\otimes R)\\
    	\Psi_{VI}^3=&2\Gamma_{\xi v}\\
    	\Phi_{VI, k}=&\delta_{\xi}\text{vecs}(P_k)
    \end{align*}

The unknown matrices of \eqref{vilinear} are $H_k,K_k$ and $\begin{bmatrix}
    		    		E\\G_2F
    		    	\end{bmatrix}$.
By iteratively applying the solution of \eqref{vilinear} to Algorithm \ref{vialg1}, we can  obtain the approximate optimal control gain.

    \begin{rem}
   As pointed out at the beginning of this section,  	reference \cite{gao2021reinforcement}  presented the VI-based method for the cooperative output regulation problem for $N$ single-input and single-output linear systems.   When $N=1$, the VI-based method in \cite{gao2021reinforcement} leads to a sequence of  equations similar to \eqref{vilinear}. However,  our derivation applies to multi-input, multi-output systems with $D \neq 0$.
     \end{rem}

    \subsection{An Improved VI-based Data-driven Algorithm} \label{sec3-2}

    In this subsection, we will  improve VI-based data-driven algorithm presented in the last subsection.

    First, let us observe from \eqref{augsys} that the matrix $P^*$ and hence the stabilizing feedback gain $K^*$ only depend on the matrices $Y$ and $J$ and they have nothing to do with the matrix
$\begin{bmatrix}
				E\\G_2F
			\end{bmatrix}$

Thus,   during the learning phase, we can  replace the  dynamic compensator $\dot{z}={G}_1z+{G}_2e$ by the following dynamic compensator \eqref{dycom}:
    \begin{align}\label{dycom}
    	\dot{{\hat{z}}}=G_1\hat{z}+G_2y
    \end{align}

 Let  $\hat{\xi}= \col (x,\hat{z})$. Then, we have
    \begin{align}
    	\dot{\hat{\xi}}&=\begin{bmatrix}
    		A&0\\G_2C&G_1
    	\end{bmatrix}\hat{\xi}+\begin{bmatrix}
    	B\\G_2D
    	\end{bmatrix}u+\begin{bmatrix}
    	E\\0
    	\end{bmatrix}v\notag \\
    	&=Y\hat{\xi}+Ju+\begin{bmatrix}
    		E\\0
    	\end{bmatrix}v
    \end{align}

    Thus, we have
    \begin{align}\label{intvinew}
    	|\hat{\xi}&(t+\delta t)|_{P_k}-|\hat{\xi}(t)|_{P_k}\notag \notag\\=&\int_{t}^{t+\delta t}\lbrack |\hat{\xi}|_{H_k}+2u^TJ^TP_k\hat{\xi}+ 2v^T\begin{bmatrix}
    		E\\0
    	\end{bmatrix}^TP_k\hat{\xi} \rbrack d\tau
    \end{align}

    From \eqref{newdefi} and \eqref{intvinew}, we have

    \begin{align}\label{vilinear2}
    	\hat{\Psi}_{VI}\begin{bmatrix}
    		\text{vecs}(H_k)\\
    		\text{vec}(J^TP_k) \\
    		\text{vec}(\begin{bmatrix}
    			E\\0
    		\end{bmatrix}^TP_k)
    	\end{bmatrix}=\hat{\Phi}_{VI, k}
    \end{align}
    where
    \begin{align*}
    	\hat{\Psi}_{VI}=&[	\hat{\Psi}_{VI}^1, 	\hat{\Psi}_{VI}^2, \hat{\Psi}_{VI}^3]\\
    		\hat{\Psi}_{VI}^1=&I_{\hat{\xi}\hat{\xi}}\\
    	\hat{\Psi}_{VI}^2=&2\Gamma_{\hat{\xi} u}\\
    	\hat{\Psi}_{VI}^3=&2\Gamma_{\hat{\xi} v}\\
    		\hat{\Phi}_{VI, k}=&\delta_{\hat{\xi}}\text{vecs}(P_k)
    \end{align*}

     The solvability of \eqref{vilinear2} is guaranteed by the following lemma.
    \begin{lem} \label{lem33}
    	The matrix $\hat{\Psi}_{VI}$ has full column rank if
    	\begin{align}\label{rankconvi2}
    		\textup{rank}([I_{\hat{\xi}\hat{\xi}}, \Gamma_{\hat{\xi} u}, \Gamma_{\hat{\xi }v}])=& \frac{(n+n_z)(n+n_z+1)}{2}\notag \\
    		&+(n+n_z)(m+q)
    	\end{align}
    \end{lem}

\begin{proof}


{\color{black}
By the definition of $\hat{\Psi}_{VI}$, we have
		    	\begin{align*}
				    		\hat{\Psi}_{VI}=
				    			[I_{\hat{\xi}\hat{\xi}}~\Gamma_{\hat{\xi }u}~ \Gamma_{\hat{\xi}v}]\Theta
				    	\end{align*}
		    	where
		    	$$\Theta=
\mbox{block~diag}(I_{\frac{(n+n_z)(n+n_z+1)}{2}}, 2I_{(n+n_z)m}, 2I_{(n+n_z)q})
				    	$$ is nonsingular.
As a result, $\text{rank}(\hat{\Psi}_{VI})=\text{rank}([I_{\hat{\xi}\hat{\xi}}~\Gamma_{\hat{\xi }u}~ \Gamma_{\hat{\xi}v}])$.
Thus, \eqref{rankconvi2} implies that $\hat{\Psi}_{VI}$ has full column rank.
}
\end{proof}

    	In comparison with \eqref{vilinear},  \eqref{vilinear2} offers two new features.   First, the expression of $\hat{\Psi}_{VI}$ is slightly simpler than that of $\Psi_{VI}$. Second, the unknown matrices of  \eqref{vilinear2} are $H_k,J$ and $E$ and they only contain $\frac{(n+n_z)(n+n_z+1)}{2}+(n+n_z)m+nq$ unknown variables. As a result, the number of the unknown variables of \eqref{vilinear2} is smaller than  the number of the unknown variables of \eqref{vilinear} by $n_z q$ and the rank condition \eqref{rankconvi2} guarantees the least squares solution  $H_k,J$ and $E$ to  \eqref{vilinear2}.

    Moreover, it is possible to further reduce the computational complexity and the solvability condition of \eqref{vilinear2} by the following procedure:

   First, in \eqref{vilinear2}, letting $k=0$ gives
     \begin{align}
    	\hat{\Psi}_{VI}\begin{bmatrix}
    		\text{vecs}(H_0)\\
    		\text{vec}(J^TP_0) \\
    		\text{vec}(\begin{bmatrix}
    			E\\0
    		\end{bmatrix}^TP_0)
    	\end{bmatrix}=\hat{\Phi}_{VI, 0}
    \end{align}
    where $P_0$ can be any positive definite matrix.

    Let $P_0= \begin{bmatrix}
    		P_{01} & 0 \\
    		0   &  P_{02}
    		\end{bmatrix}$
   where $P_{01} \in \mathbb{R}^{n \times n}$ and $P_{02} \in \mathbb{R}^{n_z \times n_z}$ are two positive definite matrices.
    Then $\Gamma_{\hat{\xi }v}\text{vec}(\begin{bmatrix}
    	E\\0
    \end{bmatrix}^T P_0)=\Gamma_{xv}\text{vec}(E^T P_{01})$. Thus, we obtain
     {\color{black}\begin{align}\label{vinew11}
     	\hat{\Psi}_0\begin{bmatrix}
     		\textup{vecs}(H_0)\\ \textup{vec}(J^TP_0)\\ \textup{vec}(E^T P_{01})
     	\end{bmatrix}=\hat{\Phi}_{VI,0}
     \end{align}
    where
    \begin{align*}
    	H_0&=Y^TP_0+P_0Y\\
    	\hat{\Psi}_0&=[I_{\hat{\xi}\hat{\xi}}, 2\Gamma_{\hat{\xi} u}, 2\Gamma_{xv}]\\
    	\hat{\Phi}_0&=\delta_{\hat{\xi}}\textup{vecs}(P_0)
    \end{align*}}

The solvability of \eqref{vinew11} is guaranteed by the following lemma.
\begin{lem} \label{lem2}
	The matrix $\hat{\Psi}_{0}$ has full column rank if
	\begin{align}\label{rankconvinew1}
		\textup{rank}([I_{\hat{\xi}\hat{\xi}}, \Gamma_{\hat{\xi} u}, \Gamma_{x v}])=& \frac{(n+n_z)(n+n_z+1)}{2}\notag \\
		&+(n+n_z)m+nq
	\end{align}
\end{lem}
\begin{proof}
	The proof of this lemma is similar to that of Lemma \ref{lem33}  and is thus omitted.
\end{proof}


{\color{black}\begin{rem}
	The rank condition \eqref{rankconvinew1} is milder than the original one \eqref{rankconvi2} since the column dimension of the matrix to be tested decreases by $n_zq$. This improvement can be significant when  $n_z q$  is large.
\end{rem}}

\begin{rem}
	If $D=0$ as in \cite{liu2018adaptive} and \cite{gao2021reinforcement}, the rank condition \eqref{rankconvinew1} can be further relaxed to
	\begin{align}\label{rankconvinew2}
		\textup{rank}([I_{\hat{\xi}\hat{\xi}}, \Gamma_{xu}, \Gamma_{x v}])=& \frac{(n+n_z)(n+n_z+1)}{2}\notag \\
		&+nm+nq
	\end{align}
	since $J=\begin{bmatrix}
		B\\0
	\end{bmatrix}$ in this case.
\end{rem}

After solving \eqref{vinew11}, we obtain the value of $J$ and $E$. Thus, the only unknown matrix in \eqref{vilinear2} is $H_k$. Since,
for any symmetric matrix $H\in \mathbb{R}^{n\times n}$, there exists a constant matrix $M_n\in \mathbb{R}^{n^2\times \frac{n(n+1)}{2}}$ with full column rank such that $M_n\text{vecs}(H)=\text{vec}(H)$, we have
\begin{align}\label{vinew111}
	\begin{split}
	\Gamma_{\hat{\xi} u}\textup{vec}(J^TP_k)&=\Gamma_{\hat{\xi} u}(I_{n+n_z}\otimes J^T)M_{n+n_z}\textup{vecs}(P_k)\\
	\Gamma_{\hat{\xi}v}\textup{vec}(\begin{bmatrix}
		E\\0
	\end{bmatrix}^TP_k)&=\Gamma_{\hat{\xi}v}(I_{n+n_z}\otimes \begin{bmatrix}
	E\\0
	\end{bmatrix}^T)M_{n+n_z}\textup{vecs}(P_k)
	\end{split}
\end{align}

From \eqref{vilinear2} and \eqref{vinew111}, we  obtain, for $k\geq 1$,
\begin{align}\label{vinew22}
\hat{\Psi}_{VI}'\textup{vecs}(H_k)=\hat{\Phi}_{VI}'\text{vecs}(P_k)
\end{align}
where
\begin{align*}
	\hat{\Psi}_{VI}'=&I_{\hat{\xi}\hat{\xi}}\\
	\hat{\Phi}_{VI}'=&\delta_{\hat{\xi}}-2\Gamma_{\hat{\xi} u}(I_{n+n_z}\otimes J^T)M_{n+n_z}\\&-2\Gamma_{\hat{\xi}v}(I_{n+n_z}\otimes \begin{bmatrix}
		E\\0
	\end{bmatrix}^T)M_{n+n_z}
\end{align*}

The matrix  $\hat{\Psi}_{VI}'$ has full column rank if $\textup{rank}(I_{\hat{\xi}\hat{\xi}})= \frac{(n+n_z)(n+n_z+1)}{2}$.


    \begin{rem}
     Since $	\hat{\Psi}_{VI}'$ has full column rank once \eqref{rankconvinew1} is satisfied,   \eqref{rankconvinew1} is the only requirement to ensure the solvability of \eqref{vinew11} and \eqref{vinew22}.
    \end{rem}

     {\color{black}\begin{rem}
    		In our improved algorithm, for all $k \geq 1$, we have reduced the problem of solving linear equation \eqref{vilinear} to that of solving \eqref{vinew22}.
     Thus,  the number of  unknown variables is reduced by $(n+n_z)(m+q)$ and  the computational cost decreases significantly.
    \end{rem}}


 \begin{rem}\label{rem5}
	{\color{black}As pointed out in \cite{bian2016}, another advantage of VI-based method over PI-based method is that $\hat{\Psi}_{VI}'$ and $\hat{\Phi}_{VI}'$ have nothing to do with the index $k$, which means that we only need to calculate the term $(\hat{\Psi}_{VI}'^T\hat{\Psi}_{VI}')^{-1}\hat{\Psi}_{VI}'^T\hat{\Phi}_{VI}'$ once, and iteratively solve $
		\text{vecs}(H_k)$ by the least-square representation as follows.
		\begin{align}\label{vinew33}
			\text{vecs}(H_k)=(\hat{\Psi}_{VI}'^T\hat{\Psi}_{VI}')^{-1}\hat{\Psi}_{VI}'^T\hat{\Phi}_{VI}' \text{vecs}(P_k)
		\end{align}
		The value of $(\hat{\Psi}_{VI}'^T\hat{\Psi}_{VI}')^{-1}\hat{\Psi}_{VI}'^T\hat{\Phi}_{VI}'$ can be saved and repeatedly used.}
\end{rem}

    Now, by iteratively applying the solution of \eqref{vinew33} to Algorithm \ref{vialg1}, we are able to obtain the approximate optimal controller.
    Our improved VI-based data-driven algorithm is summarized as Algorithm \ref{vialg2}.

    \begin{algorithm}
    	\caption{Improved VI-based Data-driven  Algorithm}\label{vialg2}
    	\begin{algorithmic}[1]
    		\State  Apply any locally essentially bounded initial input $u^0$.
    		Collecting data starting from arbitrary $t_0$ until the rank condition \eqref{rankconvinew1} is satisfied.
    		\State {\color{black}Choose $P_0= \begin{bmatrix}
    			P_{01} & 0 \\
    			0   &  P_{02}
    		\end{bmatrix}$  where $P_{01} \in \mathbb{R}^{n \times n}$ and $P_{02} \in \mathbb{R}^{n_z \times n_z}$ are two positive definite matrices. Solve $H_0$, $J$ and $E$ from \eqref{vinew11}. $k,q\leftarrow 0$.}
    		\Loop
    		\State $\tilde{P}_{k+1}\leftarrow P_k+\epsilon_k (H_k-P_kJR^{-1}J^TP_k+Q)$
    		\If{$\tilde{P}_{k+1}\notin B_q$}
    		\State $P_{k+1}\leftarrow P_0$. $q\leftarrow q+1$.
    		\ElsIf{$|\tilde{P}_{k+1}-P_k|/\epsilon_k<\varepsilon$}
    		\State \Return $K_k=-R^{-1}J^TP_k$
    		\Else
    		\State $P_{k+1}\leftarrow \tilde{P}_{k+1}$
    		\EndIf
    		\State $k \leftarrow k+1$
    		\State Solve $H_k$ from \eqref{vinew33}.
    		\EndLoop
    		\State $K^*=[K_x^*,K_z^*]=K_k$.
    		\State Obtain the following controller
    		\begin{align}\label{controllervi}
    			u^*=K_{x}^*x+K_{z}^*z
    		\end{align}
    	\end{algorithmic}
    \end{algorithm}


 For comparison of the two algorithms in this section, let us consider a case with $n=10,m=8,p=5,q=20,n_z=50$.  TABLE \ref{tab1} and TABLE \ref{tab2} show
 the numbers of unknown variables and rank conditions of the two algorithms, respectively.  The contrast is stark.

 \begin{table}[!ht]
 	\caption{Comparison of Computational Complexity}\label{tab1}
 	\center
 	\begin{tabular}{|c|c|c|}\hline

 		Algorithm & Equation Number &Unknown variables \\ \hline
 		First Algorithm & \eqref{vilinear}	& 3510 \\  \hline
 		
 		Improved Algorithm &\eqref{vinew22} & 1830 \\ \hline
 		
 	\end{tabular}
 	
 \end{table}

  	\begin{table}[!ht]
  	\caption{Comparison of rank condition}\label{tab2}
  	\center
  	\begin{tabular}{|c|c|}\hline

  		Algorithm & Rank Condition for Data Collecting \\ \hline
  		First Algorithm & 	$
  	\textup{rank}([\Gamma_{\xi\xi}, \Gamma_{\xi u}, \Gamma_{\xi v}])= 3510$ \\ \hline
  		
  	Improved Algorithm& $\textup{rank}([\Gamma_{\hat{\xi}\hat{\xi}}, \Gamma_{\hat{\xi} u}, \Gamma_{x v}])=2510$   \\  \hline
  		
  	\end{tabular}
  	
  \end{table}


\begin{thebibliography}{99}
		
%
%
%
		
		\bibitem{vrabie2009adaptive}
		D. Vrabie, O. Pastravanu, M. Abu-Khalaf, and F. L. Lewis,
		``Adaptive optimal control for continuous-time linear systems based on policy iteration'',
		\emph{Automatica}, vol. 45, no. 2, pp. 477-484, 2009.
		
		\bibitem{jiang2012computational}
		Y. Jiang and Z. P. Jiang,
		``Computational adaptive optimal control for continuous-time linear systems with completely unknown dynamics'',
		\emph{Automatica}, vol. 48, no. 10, pp. 2699-2704, 2012.


		\bibitem{bian2016}
		T. Bian and Z. P. Jiang,
		``Value iteration and adaptive dynamic pro- gramming for data-driven adaptive  optimal control design''
		\emph{Automatica}, vol. 71, pp. 348–360, 2016.
		
%
%
%
%


		
		
		
	\bibitem{gao2016adaptive}
	W. Gao and Z. P. Jiang,
	``Adaptive dynamic programming and adaptive optimal output regulation of linear systems''
	\emph{IEEE Transactions on Automatic Control}, vol. 61, no. 12, pp. 4164-4169, 2016.
	
	
	
	\bibitem{liu2018adaptive}
	Y. Liu and W. Gao,
	``Adaptive optimal output regulation of continuous-time linear systems via internal model principle'',
	in \emph{2018 9th IEEE Annual Ubiquitous Computing, Electronics \& Mobile Communication Conference (UEMCON)}, 2018, pp. 38-43.
	
	
	\bibitem{jiang2022value}
	Y. Jiang, W. Gao, J. Na, D. Zhang, T. T. Hämäläinen, V. Stojanovic and F. L. Lewis,
	``Value iteration and adaptive optimal output regulation with assured convergence rate'',
	\emph{Control Engineering Practice}, vol. 121, pp. 105042, 2022.
	
	
	\bibitem{gao2021reinforcement}
	W. Gao, M. Mynuddin, D. C. Wunsch and Z. P. Jiang,
	``Reinforcement learning-based cooperative optimal output regulation via distributed adaptive internal model'',
	\emph{IEEE transactions on neural networks and learning systems}, vol. 33, no. 10, pp. 5229-5240, 2021.
	
	\bibitem{lin2023}
L. Lin and J. Huang,
``A refined algorithm for the adaptive optimal output regulation problem'',
\emph{arXiv preprint}, arXiv:2309.15632, 2023.
	
		\bibitem{lin2024}
	L. Lin and J. Huang, ``Refined algorithms for adaptive optimal output regulation and adaptive optimal cooperative output regulation problems'', \emph{IEEE Transactions on Control of Network Systems}, 2024, DOI: 10.1109/TCNS.2024.3462549.
	


\bibitem{huang2004}
		J. Huang,
		\emph{Nonlinear Output Regulation: Theory and Applications}, Philadelphia, PA, USA: SIAM, 2004.


%
	
	\bibitem{kucera1972}
	V. Kucera,
	``A contribution to matrix quadratic equations'',
	\emph{IEEE Transactions on Automatic Control}, vol. 17, no. 3, pp. 344-347, 1972.
	
	
	
\bibitem{Kleinman} D. Kleinman, ``On an iterative technique for riccati equation computations'',
	\emph{IEEE Transactions on Automatic Control}, vol. 13, no. 1, pp. 114115, 1968.
	
	
%
%
%
%
	
	

	





	
	
	
	
	
		
		
	\end{thebibliography}
\end{document}